\documentclass{amsproc}
\usepackage{amsmath, amsthm, amssymb, amsfonts, enumitem, comment, tikz}

\usetikzlibrary{shapes.geometric, arrows.meta, positioning}

\tikzstyle{box} = [rectangle, text centered, text width=3cm, draw=black]
\tikzstyle{arrow} = [thick,->,>=stealth]

\newtheorem{theorem}{Theorem}
\numberwithin{theorem}{section}
\newtheorem{corollary}[theorem]{Corollary}
\newtheorem{lemma}[theorem]{Lemma}
\newtheorem{proposition}[theorem]{Proposition}

\theoremstyle{remark}
	\newtheorem*{remark}{Remark}
\theoremstyle{definition}
	\newtheorem{defn}[theorem]{Definition}

\newcommand{\cA}{\mathcal{A}}
\newcommand{\cB}{\mathcal{B}}

\newcommand{\cF}{\mathcal{F}}

\newcommand{\fT}{\mathfrak{T}}

\newcommand{\Z}{\mathbb{Z}} 
\newcommand{\R}{\mathbb{R}}

\newcommand{\SV}{\mathrm{SV}}
\newcommand\SVnorm[2]{\left\lVert#1\right\rVert_{\mathrm{SV}(#2)}}

\newcommand{\var}{\mathrm{var}}

\theoremstyle{plain}

\title{A Dobrushin-Lanford-Ruelle theorem for irreducible sofic shifts}

\author{Lu{\'i}sa Borsato}
\address{Departamento de Matem\'atica, Instituto de Matem\'atica e Estat\'istica, Universidade de S\~ao Paulo,
R. do Mat\~ao 1010, S\~ao Paulo, SP 05508-900, Brazil}
\email{luisabb@ime.usp.br}
\thanks{The first author is supported by grants 2018/21067-0 and 2019/08349-9, S\~ao Paulo Research Foundation (FAPESP)}

\author{Sophie MacDonald}
\address{Mathematics Department, University of British Columbia, 1984 Mathematics Road, Vancouver, British Columbia, Canada, V6T 1Z2}
\email{sophmac@math.ubc.ca}



\date{\today}

\begin{document}

\begin{abstract}
    We show that for a potential with summable variations on an irreducible sofic shift in one dimension, the equilibrium measures are precisely the shift-invariant Gibbs measures. The main tool in the proof is a preservation of Gibbsianness result for almost invertible factor codes on irreducible shifts of finite type, which we then extend to finite-to-one codes by applying the results about equilibrium measures.
\end{abstract}

\maketitle

\textit{After posting an earlier version of this preprint, we became aware of the work of Viviane Baladi \cite{baladi-1991-finitely-presented}, which obtains the same Dobrushin-Lanford-Ruelle result for H{\"older} potentials on finitely presented systems, of which irreducible sofic shifts are a particular case. We also became aware of the work of Haydn-Ruelle \cite{haydn-ruelle-1992-equivalence}, which does the same for expansive homeomorphisms with specification, of which mixing sofic shifts are a particular case. We hope that the self-contained symbolic approach of this paper, which leverages the  classical Dobrushin and Lanford-Ruelle theorems essentially as black boxes, may be accessible to a wider audience and may draw attention to the earlier work by Baladi and Haydn-Ruelle.}

\section{Introduction}

The  Dobrushin theorem establishes sufficient conditions on shift spaces $X$ and potentials $f \in C(X)$ such that every Gibbs measure for $f$ is an equilibrium measure for $f$. This theorem holds in any shift space, not necessarily of finite type, with a certain mixing condition known in the literature as condition (D) \cite{ruelle-2004-thermo}. This condition is implied, for instance, by strong irreducibility, and in this paper we only use the strongly irreducible case of the classical theorem.

The classical converse to the Dobrushin theorem is known as the Lanford-Ruelle theorem. To our knowledge, the most general natural hypothesis known for the Lanford-Ruelle theorem is the topological Markov property \cite{barbieri2018equivalence}, which is satisfied in particular by shifts of finite type. Examples are also known, however, of shift spaces which lack the topological Markov property, but for which the conclusion of the Lanford-Ruelle theorem nevertheless holds, at various levels of generality \cite{meyerovitch2013gibbs}. It is therefore desirable to extend the Lanford-Ruelle theorem beyond the class of shifts with the topological Markov property, in the hope of explaining such examples.

In this paper, we do not treat the examples in \cite{meyerovitch2013gibbs}, but we do prove a Lanford-Ruelle theorem for irreducible sofic shifts in one dimension (Theorem \ref{sofic-lf}), which generally lack the topological Markov property. This is related to a question of Kitchens-Tuncel (\cite{kitchens-tuncel-1985-sofic}, Remark 7.10(iii)). The proof relies on a preservation of Gibbsianness result for almost invertible factor codes on irreducible shifts of finite type (Proposition \ref{prop-pres-gibbs}), which we generalize to finite-to-one factor codes in Corollary \ref{irred-fin-1}. We prove Theorem \ref{sofic-lf} by lifting an equilibrium measure on a sofic shift to an equilibrium measure on a covering shift of finite type, which is Gibbs by the classical Lanford-Ruelle theorem, then concluding by Proposition \ref{prop-pres-gibbs} that the original equilibrium measure is Gibbs. Irreducibility of the sofic shift is essential: the Lanford-Ruelle theorem holds for shifts of finite type with no irreducibility assumption, but it is false in general for reducible sofic shifts. The simplest counterexample is the sunny-side-up shift (the set of sequences in $\{ 0,1 \}^{\Z}$ with at most a single $1$) with its unique shift-invariant measure.

We also extend the Dobrushin theorem to irreducible sofic shifts (Theorem \ref{sofic-dobrushin}), which generally lack the mixing properties hypothesized in the classical version. Here, our approach is based on the cyclic structure of an irreducible shift of finite type, combined with our other results.

Since the structure of this paper is somewhat involved, we outline the main components in the following flowchart. An arrow from box A to box B indicates that result A is cited in the proof of result B. The second through fourth rows consist of original results.

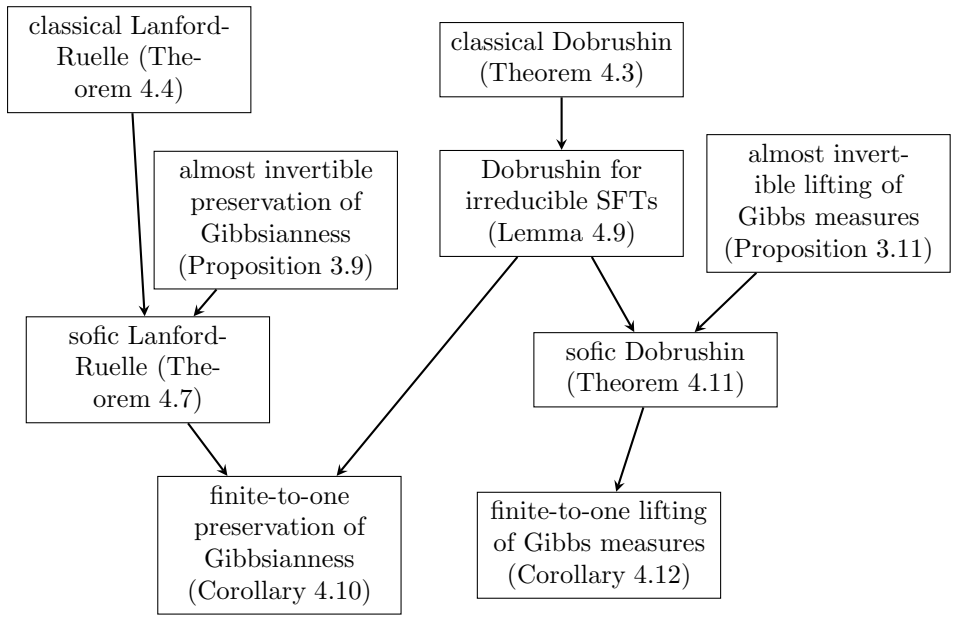
\vspace{1em}

{\centering

\begin{tikzpicture}[node distance=2cm]

\node (classic-lf) [box] {classical Lanford-Ruelle (Theorem \ref{lf})};

\node[right = 2.5cm of classic-lf] (classic-dob) [box] {classical Dobrushin (Theorem \ref{dobrushin})};

\node[below right = 0.5cm and -1.3cm of classic-lf] (ai-pres-gibbs) [box] {almost invertible preservation of Gibbsianness (Proposition \ref{prop-pres-gibbs})};

\node[below = 0.7cm of classic-dob] (sft-dob) [box] {Dobrushin for irreducible SFTs (Lemma \ref{dobrushin-irred-sft})};

\node[right = 0.3cm of sft-dob] (ai-lift-gibbs) [box]{almost invertible lifting of Gibbs measures (Proposition \ref{sofic-gibbs-pushfwd})};


\node[below right = 1cm and -2cm of sft-dob] (sofic-dob) [box] {sofic Dobrushin (Theorem \ref{sofic-dobrushin})};

\node[left= 3.5cm of sofic-dob] (sofic-lf) [box] {sofic Lanford-Ruelle (Theorem \ref{sofic-lf})};

\node[below right = 0.7cm and -1.5cm of sofic-lf] (fin-1-pres-gibbs) [box] {finite-to-one preservation of Gibbsianness (Corollary \ref{irred-fin-1})};

\node[right = 1cm of fin-1-pres-gibbs] (fin-1-lift-gibbs) [box] {finite-to-one lifting of Gibbs measures (Corollary \ref{lift-gibbs-fin-1})};


\draw [arrow] (classic-lf) -- (sofic-lf);
\draw [arrow] (classic-dob) -- (sft-dob);
\draw [arrow] (ai-pres-gibbs) -- (sofic-lf);
\draw [arrow] (ai-lift-gibbs) -- (sofic-dob);
\draw [arrow] (sft-dob) -- (sofic-dob);
\draw [arrow] (sofic-lf) -- (fin-1-pres-gibbs);
\draw [arrow] (sft-dob) -- (fin-1-pres-gibbs);
\draw [arrow] (sofic-dob) -- (fin-1-lift-gibbs);

\end{tikzpicture}
}

\section{Definitions, notations, and conventions}

\subsection{Symbolic dynamics}\label{symb-din-subsection}

Let $\cA$ be a finite set with the discrete topology, to be thought of as an alphabet, and $\cA^{\Z}$ be the \textit{full shift} with the product topology, with respect to which $\cA^{\Z}$ is compact and metrizable. The group $\Z$ acts naturally on $\cA^{\Z}$ by the shift action $\sigma$, given by $(\sigma^n x)_0 = x_n$. A \textit{shift space} is any closed, $\sigma$-invariant subset $X \subseteq \cA^{\Z}$. 

For each $n \geq 1$, we write $\cB_{n}(X)$ to denote the set of words of length $n$ in the language $\cB(X)$ of $X$---that is, the set of patterns $w \in \cA^{n}$ such that $x_{[0,n-1]} = w$ for some $x \in X$. For $w \in \cA^{n}$ we denote by $[w]_{i}$ the set of $x \in X$ with $x_{[i,i+n-1]} = w$.

We will make extensive use of continuous, shift-equivariant factor codes $\pi: X \to Y$ between shift spaces $X$ and $Y$. By the Curtis-Hedlund-Lyndon theorem, any such map $\pi$ is a sliding block code, induced by a map $\Pi: \cB_m(X) \to \cB_1(Y)$ for some $m \geq 1$. Up to a conjugacy of $X$, we can in fact assume that $\Pi$ maps symbols to symbols, i.e., $m=1$ (\cite{lind1995introduction}, Proposition 1.5.12). When $\pi: X \to Y$ is surjective, it is known as a \textit{factor} code, and $Y$ is a \textit{factor} of $X$. 

Our main results in this paper concern shifts of finite type and sofic shifts, which we now define.

\begin{defn}\label{SFT-shift}[shift of finite type]
A \textit{shift of finite type} with alphabet $\cA$ is any shift space of $\cA^{\Z}$ defined by excluding a finite number of finite words. In other words, $X \subseteq \cA^{\Z}$ is a shift of finite type if for some $n \geq 1$ there exists a finite set $\cF \subseteq \cA^n$ of forbidden words such that
\begin{equation*}
    X = \{x \in \mathcal{A}^{\Z} \,:\, \sigma^m(x)_{[0,n-1]} \notin \mathcal{F}, \mbox{ for all } m \in \Z\}
\end{equation*}
\end{defn}

\begin{defn}\label{sofic-shift}[sofic shift]
A \textit{sofic shift} is any shift space that is a continuous factor of a shift of finite type.
\end{defn}

Sofic shifts have an alternative characterization in terms of bi-infinite walks on finite edge-labeled graphs (see \cite{lind1995introduction}), but we will not need this here. Every shift of finite type is sofic, since the identity code is continuous (more generally, conjugacies preserve the class of shifts of finite type), but not every sofic shift has finite type. An example of a shift that is sofic but not of finite type is the even shift, which is the shift $X \subset \{ 0,1 \}^{\Z}$ consisting of sequences in which $10^n1$ may appear only if $n$ is even.

We will, in particular, consider measures on shift spaces, which will always be Borel probability measures. We will refer to these as $\sigma$-invariant measures to avoid any possible ambiguity, since in \S \ref{eqm-section} we also consider measures that are $\sigma^p$-invariant for some positive power $p$, but in general are not $\sigma$-invariant. In particular, we will often refer to ergodic measures, and these will always be ergodic with respect to $\sigma$.

\subsection{The Gibbs relation, cocycles, and Gibbs measures}\label{gibbs-subsec}

The \textit{Gibbs relation} on a shift space $X$, also called the tail, asymptotic, or homoclinic relation, is the equivalence relation $\fT_X \subset X \times X$ such that $(x,y) \in \fT_X$ if and only if $x_{[-N,N]^c} = y_{[-N,N]^c}$ for some $N \geq 1$. For Borel sets $A, B \subseteq X$, a \textit{holonomy} of $\fT_X$ is a Borel isomorphism $\psi: A \to B$ such that $(x, \psi(x)) \in \fT_X$ for all $x \in A$. We say that a measure $\mu$ on $X$ is $\fT_X$-nonsingular if for every Borel $A \subset X$ with $\mu(A) = 0$, we have $\mu(\fT_X(A)) = 0$, where the saturation $\fT_X(A)$ is defined as 
\[
\fT_X(A) = \{ x' \in X \, : \, \exists \, x \in A \text{ such that } (x,x') \in \fT_X      \}
\]
Note that if $\mu$ is $\fT_X$-nonsingular and $\psi: A \to B$ is a holonomy of $\fT_X$, then whenever $E \subset A$ has $\mu(E) = 0$, we have $\mu(\psi(E)) \leq \mu(\fT_X(E)) = 0$. In particular, the Radon-Nikodym derivative $\frac{d(\mu \circ \psi)}{d\mu}$ is well-defined. 

We note that $\fT_{X}$ is generated by a countable group $\Gamma$ of holonomies, in the sense that $(x,x') \in \fT_{X}$ if and only if there exists $\gamma \in \Gamma$ with $\gamma(x) = x'$. This is a special case of the main theorem of \cite{feldman-moore-1977-equivalence}. One could choose $\Gamma$ consisting of holonomies of the form $\psi_{u,v, a, b}$, where $u,v \in \cB_{b-a + 1}(X)$, for some $a, b \in \mathbb{Z}$ with $a \leq b$, and
\[
\psi_{u,v, a, b}(x) = 
\begin{cases}
x_{(-\infty, a)} u x_{(b, \infty)}, & \text{ if } x_{[a,b]} = v \text{ and }  x_{(-\infty, a)} u x_{(b, \infty)} \in X \\
x_{(-\infty, a)} v x_{(b, \infty)}, & \text{ if } x_{[a,b]} = u \text{ and }  x_{(-\infty, a)} v x_{(b, \infty)} \in X \\
x, & \text{otherwise}
\end{cases}
\] 
That is, $\psi_{u,v, a, b}$ replaces $u$ with $v$, or vice versa, whenever possible, and otherwise does nothing.

A (real, additive) \textit{cocycle} on $\fT_X$ is a Borel measurable function $\phi: \fT_X \to \R$ such that $\phi(x,y) + \phi(y,z) = \phi(x,z)$ for all $x,y,z \in X$ with $(x,y), (y,z) \in \fT_X$ (so that $(x,z) \in \fT_X$ as well). By exponentiating or taking logarithms, we can easily convert between additive and multiplicative notation for cocycles. Additive notation is more natural when the cocycle is intended to represent an energy-difference function, and multiplicative notation is more natural when the cocycle serves as a Jacobian for a change of variables.

Given a $\fT_X$-nonsingular measure $\mu$ on $X$, we say that a Borel function $D_{\mu,\fT_X}: \fT_X \to \R^+$ is a (multiplicative) \textit{Radon-Nikodym cocycle} on $\fT_X$ with respect to $\mu$ if the pushforward of $\mu$ by any holonomy $\psi: A \to B$ of $\fT_X$ satisfies $\frac{d(\mu \circ \psi)}{d\mu}(x) = D_{\mu,\fT_X}(x, \psi(x))$ for $\mu$-a.e. $x \in A$. It is routine to show that any $\fT_X$-nonsingular measure $\mu$ on $X$ has a $\mu$-a.e. unique Radon-Nikodym cocycle. 

\begin{defn}\label{Gibbs-measure-defn}[Gibbs measure]
    Let $\mu$ be a $\fT_X$-nonsingular Borel measure on a shift space $X$, and let $\phi: \fT_X \to \R$ be a cocycle. We say that $\mu$ is a \textit{Gibbs measure} if for any holonomy $\psi: A \to B$ of $\fT_X$, and $\mu$-a.e. $x \in A$, we have $D_{\mu, \fT_X}(x,\psi(x)) = \exp(\phi(x,\psi(x)))$.
\end{defn}

To put it another way, a measure is by definition Gibbs if and only if it is nonsingular, and a nonsingular measure is Gibbs precisely with respect to the logarithm of its own Radon-Nikodym cocycle. These measures are also known as conformal measures in the literature \cite{borsato2020conformal, meyerovitch2013gibbs}.

In \cite{borsato2020conformal} it is shown, building on results of Kimura \cite{kimura-2015-thesis}, Keller \cite{keller1998equilibrium}, and others, that the definition of a Gibbs measure that we have given is equivalent to another well-known one involving the Dobrushin-Lanford-Ruelle equations, in terms of which the theorems of Dobrushin (Theorem \ref{dobrushin}) and Lanford-Ruelle (Theorem \ref{lf}) were originally stated. These Gibbs measures do not coincide, in general, with Gibbs measures in the sense of Bowen; see the second remark after Proposition \ref{prop-pres-gibbs} for further discussion.

\section{Preservation of Gibbsianness}\label{pres-gibbs-sec}

In this section, we prove a pair of preservation of Gibbsianness results, namely Propositions \ref{prop-pres-gibbs} and \ref{sofic-gibbs-pushfwd}, which are essential to our main theorems, Theorem \ref{sofic-lf} and \ref{sofic-dobrushin}. 

\begin{defn}\label{defn-irred-z} [irreducibility]
A shift space $X$ is irreducible if for every ordered pair of blocks $u, v \in \cB(X)$ there is $w \in \cB(X)$ so that $uwv \in \cB(X)$.
\end{defn}

\begin{defn}\label{defn-aper}[period]
The \textit{period} of an irreducible sofic shift is the greatest common divisor of the least periods of its periodic points.
\end{defn}

\begin{defn}\label{strong-irred-defn}[strong irreducibility]
A shift space $X$ is \textit{strongly irreducible} if there exists some $r \geq 1$ such that for any $u,v \in \cB(X)$ and any $s \geq r$, there exists $w \in \cB_s(X)$ with $uwv \in \cB(X)$.
\end{defn}

\begin{remark}
A shift of finite type is strongly irreducible if and only if it is topologically mixing, if and only if it is irreducible and has period $1$.
\end{remark}

The following proposition generalizes part of the proof of Lemma 4.1 in \cite{meester-steif-2001-higher-dim}. (They treat the case of a uniform Gibbs measure on a strongly irreducible shift of finite type over $\Z^d$.) The proof is also very similar to that of Proposition 5.2 in \cite{marcus-pavlov-2013-approx}.

\begin{proposition}\label{strong-irred-support}
Let $X$ be a strongly irreducible shift space. Then any $\fT_X$-invariant nonsingular measure on $X$ has full support.
\end{proposition}

\begin{proof}
Let $\mu$ be a $\fT_X$-invariant nonsingular measure on $X$. Let $r \geq 1$ be a witness for the strong irreducibility of $X$. Fix $a, b \in \Z$ with $a < b$, and fix $w \in \cB_{b-a}(X)$. By strong irreducibility, for each $n \geq 1$ and each $p, s \in \cB_{n}(X)$, there exists $w' \in \cB_{(b-a)+2r} (X)$ such that $w'_{[a,b]} = w$ and $p w' s \in \cB(X)$.
By compactness, it follows that for every $x \in X$, there exists some $u \in \cB(X)$ with $u_{[a,b]} = w$ and $x_{(-\infty, a-r]} u x_{[b+r, \infty)} \in X$.

For each pair $u,v \in \cB_{(b-a)+2r-1}(X)$ with $u_{[a,b]} = w$, let 
\[
E_{u,v} = [v]_{a-r} \cap \{ x \in X \, : x_{(-\infty, a-r]} u x_{[b+r, \infty)} \in X  \}
\]
Then $\bigcup_{u,v} E_{u, v} = X$. In particular, $\mu(E_{u,v}) > 0$ for at least one pair $u,v$. Let $\psi_{u,v,a,b}$ be the holonomy of $\fT_X$ that exchanges $u$ and $v$ on $[a,b]$ when possible and does nothing else, as in \S \ref{gibbs-subsec}. By the definition of $E_{u,v}$, we have that for every $x \in E_{u,v}$, $\psi_{u,v,a,b}(x) \in [u]_{a-r}$. Then we have
\begin{align*}
    \mu([w]_a) &\geq \mu([u]_{a-r}) \\
    &\geq \mu(\psi_{u,v,a,b}(E_{u,v})) \\
    &= \int_{E_{u,v}} D_{\mu, \fT_X}(x, \psi_{u,v,a,b}(x)) \, d\mu(x)  \\
    &> 0
\end{align*}
Since $w$ was an arbitrary word at an arbitrary position, $\mu$ has full support.
\end{proof}

\begin{remark}
The statement and proof of Proposition \ref{strong-irred-support} generalize essentially without modification when $\Z$ is replaced by an arbitrary countable group, with a suitable generalization of strong irreducibility \cite{ceccherini-silberstein-coornaert-2012-periodic, barbieri2018equivalence}. We also mention that strong irreducibility is used in the proof of Proposition \ref{strong-irred-support} to show that all $\fT_X$ equivalence classes are dense in $X$. These equivalence classes are the orbits of the action of the countable group $\Gamma$ generating $\fT_X$. In the special case of a shift of finite type, $\Gamma$ can be taken to be generated by homeomorphisms \cite{meyerovitch2013gibbs}. We can thus interpret Proposition \ref{strong-irred-support} as the statement that a nonsingular measure for a minimal continuous action must have full support.
\end{remark}

\begin{defn}\label{defn-doub-trans} [doubly transitive point]
Let $X$ be a shift space. A point $x \in X$ is \textit{doubly transitive} if every word $w \in \cB(X)$ appears in $x$ infinitely often to the left and to the right.
\end{defn}

It is easy to check that a shift space $X$ contains a doubly transitive point if and only if $X$ is irreducible (see \S 9.1, \cite{lind1995introduction} for more details). The following lemma will also be useful.

\begin{lemma}\label{lemma-doub-trans-invar}
Let $X$ be a shift space. Then the set $D_X \subset X$ of doubly transitive points is $\fT_X$-invariant.
\end{lemma}

\begin{proof}
If $X$ is not irreducible then $D_X = \emptyset$, which is trivially $\fT_X$-invariant, so assume that $X$ is irreducible. Let $x \in D_X$ and suppose that $x' \in X$ with $(x,x') \in \fT_X$. Then there exists some $\Delta = [a,b] \cap \Z$ such that $x_{\Delta^c} = x'_{\Delta^c}$. Since $x \in D_X$, every word $w \in \cB(X)$ appears infinitely often both in $x_{(-\infty, a-1]} = x'_{(-\infty, a-1]}$ and in $x_{[b+1, \infty)} = x'_{[b+1, \infty)}$. Therefore $x' \in D_X$.
\end{proof}

It is then immediate that $X \setminus D_X$ is also $\fT_X$-invariant. (In general, if $\mathcal{R}$ is an equivalence relation on $X$ and $A$ is an $\mathcal{R}$-invariant subset, then $A^c$ is also $\mathcal{R}$-invariant.)

The following is a generalization of Theorem 9.4.9 in \cite{lind1995introduction}, which appears to be well-known (\cite{yoo2018decomposition}, proof of Lemma 4.5). We thank Tom Meyerovitch for this short proof.

\begin{proposition}\label{prop-doub-trans-full-meas}
Let $X$ be an irreducible shift space and let $\mu$ be a fully supported $\sigma$-invariant ergodic  measure on $X$. Let $D_X$ denote the set of doubly transitive points. Then $\mu(D_X) = 1$.
\end{proposition}

\begin{proof}
Recall that a \textit{generic point}, for a continuous transformation of a compact space with an $\sigma$-invariant measure, is a point that satisfies the conclusion of the pointwise ergodic theorem (Theorem $1.14$, \cite{walters1982ergodic}) for every continuous function. Since $\mu$ is ergodic with respect to $\sigma$, the sets of generic points with respect to $\sigma$ and $\sigma^{-1}$ have full measure (\cite{furstenberg-1981-recurrence}, Proposition 3.7), so their intersection has full measure as well. Since $\mu$ has full support, every point that is generic for both $\sigma$ and $\sigma^{-1}$ (in particular, almost every point in $X$) is doubly transitive. In greater detail: we prove the contrapositive. Suppose $x$ is not doubly transitive. Then there is a word $w$ which appears at most finitely often, without loss of generality to the right, in $x$. Then
\[
\lim_{n \to \infty} \frac{1}{n} \sum_{k=0}^{n-1} \mathbf{1}_{[w]_0} (\sigma^k x) = 0 \neq \mu([w]_0)
\]
so $x$ does not satisfy the conclusion of the ergodic theorem for the continuous function $\mathbf{1}_{[w]_0}$. Therefore $x$ is not generic for $\sigma$.
\end{proof}

Let $X$ be a shift of finite type, let $Y$ be a sofic shift, and let $\pi: X \to Y$ be a finite-to-one factor code. Then there is an integer $d \geq 1$, known as the \textit{degree} of $\pi$, such that each doubly transitive point $y \in Y$ has exactly $d$ $\pi$-preimages \cite{lind1995introduction}. An important special case is when the degree is one, which is in fact equivalent to the following condition, which a priori might seem more general.

\begin{defn}\label{almost-invertibility}[almost invertibility]
Let $X$ be an irreducible shift of finite type, let $Y$ be an irreducible sofic shift, and let $\pi: X \to Y$ be a factor code. We say that $\pi$ is \textit{almost invertible} if every doubly transitive point $y \in D_Y$ has a unique preimage.
\end{defn}

\begin{remark} A factor code on an irreducible shift of finite type is almost invertible if and only if it is finite-to-one with degree one (\cite{lind1995introduction}, Proposition 9.2.2).
\end{remark}

The following is our first preservation of Gibbsianness result.

\begin{proposition}[preservation of Gibbsianness for almost invertible factor codes]\label{prop-pres-gibbs}
Let $X$ be an irreducible shift of finite type, let $Y$ be a sofic shift, and let $\pi: X \to Y$ be an almost invertible factor code. Let $\mu$ be a measure on $X$ which is fully supported, $\sigma$-invariant, ergodic, and $\fT_X$-nonsingular. Let $\nu = \pi_* \mu$. Then $\nu$ is $\fT_Y$-nonsingular. Moreover, if $\mu$ is Gibbs with respect to an additive cocycle $\phi$, then $\nu$ is Gibbs with respect to $\phi \circ (\pi^{-1} \times \pi^{-1})$, where $\pi^{-1}$ is well-defined $\nu$-almost everywhere.
\end{proposition}

\begin{proof}
First, for a technical reason described below, we assume that $\pi$ is a one-block code, in the sense that it is induced by a map $\Pi\colon \cB_1(X) \to \cB_1(Y)$. We assume further that it has a \textit{magic symbol}---that is, a symbol $b \in \cB_1(Y)$ with a unique $\Pi$-preimage $a \in \cB_1(X)$, such that if $\pi(x)=y$ and $y_0 = b$, then $x_0 = a$. These assumptions incur no loss of generality (\cite{lind1995introduction}, \S9.1).

These assumptions made, we see how to lift holonomies of $\fT_Y$ to holonomies of $\fT_X$. Since $\pi$ is a Borel map between complete separable metric spaces $X,Y$ which restricts to an injection on the Borel set $D_X$, the inverse $\pi|_{D_X}^{-1}: D_Y \to D_X$  is Borel (\cite{kechris1995descriptive}, Theorem 15.1). Now let $\psi: A \to B$ be a holonomy of $\fT_Y$. Define the map $\Tilde{\psi} : \pi^{-1}(A) \cap D_X \to \pi^{-1}(B) \cap D_X$ by $\Tilde{\psi} = \pi^{-1} \circ \psi \circ \pi$; this is clearly a measurable bijection. 

Moreover, to see that $(x,\Tilde{\psi}(x)) \in \fT_X$  for each $x \in \pi^{-1}(A) \cap D_X$, observe that $(\pi(x), \pi \circ \Tilde{\psi}(x)) = (\pi(x), \psi \circ \pi(x)) \in \fT_Y$, so there exist $m < n$ with $\pi(x)_{[m,n]^c} = \psi \circ \pi(x)_{[m,n]^c}$. Since $\pi(x)$ and $\psi \circ \pi (x)$ are doubly transitive, the magic symbol $b$ for $\pi$ appears infinitely often to the left and right. We may therefore assume, by taking $m$ and $n$ larger if necessary, that
\[
\pi(x)_{m-1} = \psi \circ \pi(x)_{m-1} = b = \pi(x)_{n+1} = \psi \circ \pi(x)_{n+1}
\]
Moreover, since the magic symbol occurs infinitely often to the left and right, any word outside $[m,n]$ appears within a word beginning and ending with $b$. Proposition 9.1.9 in \cite{lind1995introduction} asserts, in the almost invertible case, that such a word has a unique $\pi$-preimage. This shows that $(x,\Tilde{\psi}(x))$ agree outside of $[m,n]$, so $\Tilde{\psi}$ is indeed a holonomy of $\fT_X$.

Now, let $\Gamma_X, \Gamma_Y$ be countable groups of holonomies generating $\fT_X, \fT_Y$ respectively, and let $A \subset Y$ be Borel with $\nu(A) = 0$. Observe that $\fT_Y(A) = \bigcup_{\gamma \in \Gamma_Y} \gamma(A)$. Observe further that, for each $\gamma \in \Gamma_Y$,
\begin{align*}
    \nu(\gamma(A)) &= \nu( \gamma(A) \setminus D_Y) + \nu( \gamma(A) \cap D_Y   ) \\
    &= \mu\left(\pi^{-1}(\gamma(A) \cap D_Y)\right)\\
    &= \mu( \Tilde{\gamma} (\pi^{-1}(A \cap D_Y))          ) \\
    &= 0
\end{align*}
since $\pi^{-1}(A \cap D_Y)$ is $\mu$-null and $\mu$ is $\fT_X$-nonsingular. Therefore $\nu$ is indeed $\fT_Y$-nonsingular.

Again, let $\psi: A \to B$ be a holonomy of $\fT_Y$ and let $\Tilde{\psi}$ be as above. Then
\begin{align*}
    \nu(B) &= \mu( \pi^{-1}(\psi(A \cap D_Y))     )     \\
    &= \mu(  \Tilde{\psi}(\pi^{-1} (A) \cap D_X )     ) \\
    &= \int_{\pi^{-1} (A) \cap D_X} D_{\mu, \fT_X}( x, \Tilde{\psi}(x)   ) \, d\mu(x) \\
    &= \int_{A \cap D_Y} D_{\mu, \fT_X}( \pi^{-1}(y), \pi^{-1}(\psi(y))  ) \, d\nu(y),
\end{align*}
where the last equality follows from the change of variables formula.

On the other hand, since $\nu$ is $\fT_Y$-nonsingular, we know that
\[
\nu(B) = \int_A D_{\mu, \fT_Y}( y, \psi(y)  ) \, d\nu(y)
\]
By the uniqueness of the Radon-Nikodym derivative, we then have, for almost all $(y,y') \in D_Y^2 \cap \fT_Y$, that $D_{\nu, \fT_Y}(y,y') = D_{\mu, \fT_X}( \pi^{-1}(y), \pi^{-1}(y')  )$. 
\end{proof}

\begin{remark}
In Corollary \ref{irred-fin-1}, we generalize Proposition \ref{prop-pres-gibbs} from almost invertible to finite-to-one codes, in the case that the cocycle on the range is induced by a sufficiently regular potential. It is therefore natural to ask whether the proof of Proposition \ref{prop-pres-gibbs} can be adapted to the finite-to-one setting. However, such an adaptation would not be straightforward, because at higher degrees we lose a key condition on which the proof of Proposition \ref{prop-pres-gibbs} relied: namely, that preimages of asymptotic doubly transitive points must themselves be asymptotic. Indeed, if $\pi$ has degree $d > 1$, then every doubly transitive point $y \in Y$ has $d$ preimages, no two of which ever exhibit the same symbol in the same position (see Exercise 9.1.3, \cite{lind1995introduction}), and are therefore about as far from asymptotic as one could imagine.
\end{remark}

\begin{remark}
The hypotheses of Proposition \ref{prop-pres-gibbs} are likely more restrictive than would be needed simply to show that the pushforward of a $\fT_X$-nonsingular measure is $\fT_Y$-nonsingular. The reason is that, for the application in Theorem \ref{sofic-lf}, we need pointwise control over the potential inducing the Radon-Nikodym cocycle of the pushforward measure, whereas for the purposes of \cite{pollicott-kempton-2011-factors}, for instance, it is sufficient to determine the potential's regularity. For instance, consider a very simple symbol amalgamation code from the full $3$-shift $X= \{ 0,1,2  \}^{\Z}$ to the full $2$-shift $Y=\{ 0,1 \}^{\Z}$, given by amalgamating the symbols $1,2$ into the symbol $1$. This code takes the uniform Bernoulli measure on $X$, which is Gibbs for the zero cocycle, to the $(1/3,2/3)$ Bernoulli measure on $Y$, which is Gibbs with respect to a cocycle obtained from a locally constant potential. This would count as preservation of Gibbsianness in the sense of \cite{pollicott-kempton-2011-factors} but not in ours.
\end{remark}

We now proceed toward Proposition \ref{sofic-gibbs-pushfwd}, which is a converse result to Proposition \ref{prop-pres-gibbs}, showing that every nonsingular measure on an irreducible sofic shift is the pushforward, through an almost invertible code, of a nonsingular measure on an irreducible shift of finite type. We begin with the following lemma, which is well-known, and can be regarded as a relative version of the Krylov-Bogliubov theorem. One standard proof uses the Hahn-Banach theorem. For completeness, and possibly independent interest, we include a different proof, based on the standard proof of the (non-relative) Krylov-Bogliubov theorem \cite{walters1982ergodic}. The Hahn-Banach argument requires no dynamical assumptions at all, whereas our argument relies on the dynamical setting to make the argument somewhat more constructive. We state it only for ergodic measures on shift spaces; the argument goes through for any continuous transformation of a compact metric space, and it easily generalizes to any $\sigma$-invariant measure by convexity.

\begin{lemma}\label{rel-kry-bog}
Let $X$ and $Y$ be shift spaces, let $\pi: X \to Y$ be a factor code, and let $\nu$ be a $\sigma$-invariant measure on $Y$. Then there exists a $\sigma$-invariant measure $\mu$ on $X$ with $\pi_* \mu = \nu$. If $\nu$ is ergodic then $\mu$ can be chosen to be ergodic as well.
\end{lemma}

\begin{proof}
Let $y \in Y$ be a generic point with a preimage $x \in X$, and let $\nu_n = \frac{1}{n} \sum_{k=0}^{n-1} \delta_y \circ \sigma^{-k}$ be the $n^{\text{th}}$ empirical measure for $y$. By the ergodic theorem, $\nu_n$ converges to $\nu$ in the weak-$*$ topology. Similarly, let $\mu_n = \frac{1}{n} \sum_{k=0}^{n-1} \delta_x \circ \sigma^{-k}$, and appeal to compactness to extract a weak-$*$ convergent subsequence with limit $\mu$. It is not hard to see that $\mu$ must be $\sigma$-invariant (since $\| \mu_n - \mu_n \circ \sigma   \|_{\mathrm{TV}} \leq 2/n$). Moreover, since $\nu_n = \pi_* \mu_n$ and the pushforward operation is continuous, we can conclude that $\pi_* \mu = \nu$. The fact that $\mu$ can be chosen to be ergodic follows by a standard convexity argument (see \cite{simon2011convexity}, Chapter 8).
\end{proof}

\begin{proposition}[lifting Gibbs measures through almost invertible factor codes]\label{sofic-gibbs-pushfwd}
Let $X$ be an irreducible shift of finite type, let $Y$ be a sofic shift, and let $\pi: X \to Y$ be an almost invertible factor code. Let $\nu$ be a measure on $Y$ which is fully supported, $\sigma$-invariant, ergodic, and $\fT_Y$-nonsingular. If $\nu$ is Gibbs for an additive cocycle $\phi$, then every $\sigma$-invariant measure $\mu$ on $X$ with $\pi_* \mu = \nu$ is  $\fT_X$-nonsingular and, in particular, is a $\sigma$-invariant Gibbs measure for $\phi \circ (\pi \times \pi)$.
\end{proposition}

\begin{proof}
As in the proof of Proposition \ref{prop-pres-gibbs}, we assume without loss of generality (that is, up to a conjugacy of $X$) that $\pi$ is a one-block code, induced by a block map $\Pi: \cB_1(X) \to \cB_1(Y)$. We assume further that $\pi$ has a magic symbol $b \in \cB_1(Y)$, which has a unique $\Pi$-preimage $a \in \cB_1(X)$.

Let $\mu$ be a $\sigma$-invariant measure on $X$ with $\pi_* \mu = \nu$; such a $\mu$ exists by Lemma \ref{rel-kry-bog}. We first suppose further that $\mu$ is ergodic. We begin by showing that $\mu$ has full support. Let $w \in \cB(X)$ be arbitrary. By irreducibility, there exist $u,v \in \cB(X)$ such that $auwva \in \cB(X)$. Let $s = \Pi(uwv)$, so that, by the same argument as in the proof of Proposition \ref{prop-pres-gibbs} (using \cite{lind1995introduction}, \S 9.1), $[auwva]_0 = \pi^{-1}([bsb]_0)$. Since $\nu$ has full support by hypothesis, $\nu([bsb]_0) > 0$, so $\mu([w]_0) > 0$, since $[auwva] \subseteq [w]$, with coordinates lined up appropriately. Since $w$ was arbitrary and $\mu$ is $\sigma$-invariant, $\mu$ has full support.

Let $\psi\colon X \to X$ be a holonomy of $\fT_X$. As in the proof of Proposition \ref{prop-pres-gibbs} (but with the holonomies being pushed in the opposite direction), $\pi$ forms a Borel isomorphism between $D_X$ and $D_Y$, and if $x,x' \in D_X$ then $(x,x') \in \fT_X$ if and only if $(\pi(x),\pi(x')) \in \fT_Y$. Therefore, we have a holonomy $\Tilde{\psi} = \pi \circ \psi \circ (\pi|_{D_X})^{-1} : D_Y \to D_Y$ of $\fT_Y$.

Let $N \subset X$ be a Borel set with $\mu$-measure zero. Observe that 
\[
\psi(N) = \psi(N \cap D_X) \cup \psi(N \setminus D_X) 
\]

This yields that $\mu(\psi(N)) = \mu(\psi(N) \cap D_X)$, since $\psi(N \setminus D_X) = \psi(N)\setminus D_X$, and $D_X$ has full measure by ergodicity and Proposition \ref{prop-doub-trans-full-meas}. Moreover, 
\[
\psi(N \cap D_X) = \pi^{-1} \circ \Tilde{\psi} ( \pi(N) \cap D_Y )
\]
Now, $\nu(\pi(N) \cap D_Y) = \mu(N \cap D_X) = 0$, so $\nu(\Tilde{\psi} ( \pi(N) \cap D_Y ))$ = 0, since $\nu$ is $\fT_Y$-nonsingular. Therefore $\mu(\psi(N \cap D_X)) = 0$ as well, so in fact $\mu(\psi(N)) = 0$, which shows that $\mu$ is indeed 
$\fT_X$-nonsingular. A calculation in the same spirit, very similar to the calculation that concludes the proof of Proposition \ref{prop-pres-gibbs}, shows that $\mu$ is Gibbs for $\phi \circ (\pi \times \pi)$, as claimed. This concludes the proof for $\mu$ ergodic; the general case follows by convexity.
\end{proof}

\section{Equilibrium measures}\label{eqm-section}

\begin{defn}\label{topological-pressure-and-equilibrium-states}[topological pressure and equilibrium states]
Let $X$ be a shift space and let $f\colon X \to \R $ be a continuous function. The \textit{topological pressure} of $f$ is the value
\[
P_X(\sigma, f) = \sup \left\{ h(\mu) + \int_X f \, d\mu\right\}
\]
where the supremum is over all $\sigma$-invariant measures $\mu$ on $X$. Any measure $\mu$ attaining the supremum is known as an \textit{equilibrium measure} for $f$.
\end{defn}

\begin{remark}
The pressure is sometimes defined as above (see \cite{keller1998equilibrium}), but is often defined differently (see e.g. \cite{walters1982ergodic}), with the variational property by which we defined it stated as a theorem. However, this variational property is the only one we will need, apart from the two classical theorems below.
\end{remark}

\begin{defn}\label{SV-space}[the function space $\SV(X)$]
Let $X$ be a shift space and let $f\colon X \to \R $ be a continuous function. We define the \textit{$k^{\text{th}}$ variation} of $f$ as
\[
v_k(f) = \sup \{ | f(x) - f(y) | \, : \, x_{[-k,k]} = y_{[-k,k]} \}
\]
for $k \geq 0$; it is convenient to define $\var_{-1}(f) = \| f \|_{\infty}$. We then define
\[
\SVnorm{f}{X} = \sum_{k=0}^{\infty}  v_{k-1}(f)
\]
and define the class of potentials $\SV(X) = \{ f \in C(X) \, : \,  \SVnorm{f}{X} < \infty   \}$.
\end{defn}

It is easy to verify that $\left(\SV(X), \SVnorm{\cdot}{X} \right)$ is a Banach space, and that (\cite{meyerovitch2013gibbs}, \cite{borsato2020conformal}) a potential $f \in \SV(X)$ defines a cocycle $\phi_f$ on $\fT_X$ via the following absolutely convergent series:
\[
\phi_f(x,y) = \sum_{n \in \Z} [ f(\sigma^n y) - f(\sigma^n x)  ]
\]

We refer to a Gibbs measure for $\phi_f$ simply as a Gibbs measure for $f$. With this definition, we recall the following classical theorems, which we state only for the special cases that we require. These theorems were originally proved in a somewhat different form, using the formalism of interactions rather than potentials as we have used, but the methods adapt easily.

\begin{theorem}[Dobrushin; see \cite{ruelle-2004-thermo}]\label{dobrushin}
Let $X$ be a strongly irreducible shift space and let $f \in \SV(X)$. Every $\sigma$-invariant Gibbs measure for $f$ is an equilibrium measure for $f$.
\end{theorem}

\begin{theorem}[Lanford-Ruelle; see \cite{meyerovitch2013gibbs}]\label{lf}
Let $X$ be a shift of finite type and let $f \in \SV(X)$. Every equilibrium measure for $f$ is a Gibbs measure for $f$.
\end{theorem}

We also require the following result, closely following Lemma 4.5 in \cite{yoo2018decomposition}.

\begin{proposition}\label{yoo-sv}
Let $X$ be an irreducible shift of finite type. Any equilibrium measure for $f \in \SV(X)$ has full support.
\end{proposition}

\begin{proof}
When $X$ is a mixing shift of finite type, the result follows from Theorem \ref{lf} and Proposition \ref{strong-irred-support}. Now, let $X$ be an irreducible shift of finite type such that $X$ has period $p$. We decompose $X$ into $p$ cyclically moving classes (see \cite{lind1995introduction}, \S 4.5), i.e., $X = \sqcup_{i=0}^{p-1} X_i$ where $\sigma(X_i) = X_{i+1 \mod p}$ and each $X_i$ is mixing with respect to $\sigma^p$. In particular, $\sigma^p$ is a homeomorphism of each clopen set $X_i$; the system $(X, \sigma^p)$ is known as the $p^{\text{th}}$ higher power shift of $X$ (\cite{lind1995introduction}, \S 1.4), and the $X_i$ are its cyclically moving classes. 

There is a bijection between $\sigma$-invariant measures $\mu$ on $X$ and $\sigma^p$-invariant measures $\mu'$ on $X_0$, which is given by normalized restriction in one direction and averaging in the other. That is, we take $\mu' = p \mu|_{X_0}$ and $\mu = p^{-1 }\sum_{j=0}^{p-1} \sigma^{j}_* \mu'$. Observe that $\mu$ has full support if and only if $\mu'$ does. 
Moreover, we have
\[\label{eq-pressure-mixing-comp}
h_X(\mu) + \int_X f \, d\mu = \frac{1}{p} \left( h_{X_0}(\mu') + \int_{X_0} R_p f \, d\mu'         \right)
\]
where $R_p f =  \sum_{j=0}^{p-1} f \circ \sigma^{-j}$. Therefore $\mu$ is an equilibrium measure for $f$ on $X$ precisely when $\mu'$ is an equilibrium measure for $R_p f$ on $X_0$. So far, our discussion is essentially identical to Yoo's.

We now need to show that $R_p f \in \SV(X_0)$; this is where we part from Yoo, who considers a different class of potentials.  Observe that, for each $j \in \{0, \dots, p-1\}$, the $k$-th variation of a function on $X_0$ behaves like the $kp$-th variation of that function on $X$, so we have
\[
\SVnorm{f \circ \sigma^{-j}}{X_0} \leq \sum_{k=0}^{\infty} v_{kp}(f \circ \sigma^{-j}) \leq \frac{1}{p} \SVnorm{f \circ \sigma^{-j}}{X}
\]
It is also easy to check that $\SV(X)$ is closed under translation, showing that indeed $R_p f \in \SV(X_0)$. 

Let $\mu$  be an equilibrium measure for $f$. Then the unique $\mu'$ on $X_0$ such that $\mu = p^{-1 }\sum_{j=0}^{p-1} \sigma^{j}_* \mu'$ is an equilibrium measure for $R_p f$; since $X_0$ is mixing, $\mu'$ is fully supported. Thus $\mu$ is fully supported as well.
\end{proof}

We also require a result showing that equilibrium measures lift through almost invertible codes, indeed through any finite-to-one codes.
\begin{lemma}\label{lift-eqm}
Let $X$ be an irreducible shift of finite type, $Y$ a sofic shift, and $\pi: X \to Y$ a finite-to-one factor code. Let $f\colon Y \to \R$ be a function with summable variation, that is, $f \in \SV(Y)$, and let $\nu$ be an equilibrium measure for $f$. Then there exists an equilibrium measure $\mu$ for $f \circ \pi \in \SV(X)$ with $\pi_* \mu = \nu$. If $\nu$ is ergodic, then $\mu$ can be chosen to be ergodic as well.
\end{lemma}
\begin{proof}
By Lemma \ref{rel-kry-bog}, there exists a $\sigma$-invariant measure $\mu$ on $X$ with $\pi_* \mu = \nu$, and we can choose $\mu$ to be ergodic whenever $\nu$ is ergodic. We now show that any such $\mu$ is an equilibrium measure for $f \circ \pi$. Since $\pi$ is finite-to-one, the Abramov-Rokhlin formula \cite{adler1963skew} shows that $h(\mu) = h(\nu)$. Then
\[
    h(\mu) + \int_X f \circ \pi \, d\mu = h(\nu) + \int_Y f \, d\nu = P_Y(\sigma, f)
\]
We need to show that $P_X(\sigma, f \circ \pi) = P_Y(\sigma, f)$. Clearly $P_X(\sigma, f \circ \pi) \geq P_Y(\sigma, f)$, so we only need the reverse inequality. Let $\lambda$ be any $\sigma$-invariant measure on $X$. Again $h(\pi_* \lambda) = h(\lambda)$ by the Abramov-Rokhlin formula, so we have
\[
    h(\lambda) + \int_X f \circ \pi \, d\lambda = h(\pi_* \lambda) + \int_Y f \, d\pi_* \lambda \leq P_Y(f)
\]
Therefore $P_X(\sigma, f \circ \pi) = P_Y(\sigma, f)$, so $\mu$ is indeed an equilibrium measure for $f \circ \pi$.
\end{proof}

We can now state and prove the first main result of this section.

\begin{theorem}[Lanford-Ruelle theorem for irreducible sofic shifts]\label{sofic-lf}
Let $Y$ be an irreducible sofic shift and let $f\colon Y \to \R$ be a potential with summable variation, that is, $f \in SV(Y)$. Let $\nu$ be an equilibrium measure for $f$. Then $\nu$ is a Gibbs measure for $f$. 
\end{theorem}

\begin{proof}
Since $Y$ is an irreducible sofic shift, there exist an irreducible shift of finite type $X$ and an almost invertible factor code $\pi\colon X \to Y$, for instance the minimal right-resolving presentation (\cite{lind1995introduction}, \S 3.3, \S 9.2). 

Suppose first that $\nu$ is ergodic. By Lemma \ref{lift-eqm}, there exists an ergodic equilibrium measure $\mu$ for $f \circ \pi$ such that $\nu = \pi_{*}\mu$.  Then, by Theorem \ref{lf}, $\mu$ is a Gibbs measure for $f \circ \pi$, and by Proposition \ref{yoo-sv}, $\mu$ has full support.  Then, by Proposition \ref{prop-pres-gibbs}, which requires ergodicity and full support, $\nu$ is a Gibbs measure for $f$.

The general result follows from the ergodic case via the Krein-Milman theorem \cite{simon2011convexity}, together with the compactness and convexity of the sets of Gibbs and equilibrium measures \cite{keller1998equilibrium} and the fact that the extreme points of these sets are precisely their ergodic elements.
\end{proof}

We now change course and proceed towards a Dobrushin type theorem for irreducible sofic shifts, which is the second main result of this section.

\begin{lemma}\label{lemma-period-trans}
Let $X$ be an irreducible shift of finite type of period $p$, partitioned into $p$ cyclically moving classes $X_i$, $0 \leq i \leq p -1$. Then any $\fT_X$-equivalence class is contained in a single cyclically moving class. That is, if $x \in X_i$ and $x' \in X_j$ with $(x,x') \in \fT_X$, then $i=j$.
\end{lemma}

\begin{proof}
Suppose without loss of generality that $X$ is an edge shift with alphabet $\cA = \cB_1(X)$. Then $\cA$ can be partitioned into $p$ subsets  $\cA_i$, $0 \leq i \leq p-1$, corresponding to the cyclically moving classes $X_i$, such that for any $x \in X$, $x \in X_i$ if and only if $x \in \cA_i$. Thus for $x,x' \in X$, if $x_n = x'_n$ for some $n$, then $x,x'$ are in the same cyclically moving class. But if $(x,x') \in \fT_X$ then $x_n = x'_n$ for all but finitely many $n$, so in particular $x$ and $x'$ are in the same cyclically moving class.
\end{proof}

\begin{lemma}[Dobrushin theorem for irreducible shifts of finite type]\label{dobrushin-irred-sft}
Let $X$ be an irreducible shift of finite type and $f\colon X \to \R$  be a potential with summable variation, that is $f \in SV(X)$. Then every $\sigma$-invariant Gibbs measure for $f$ is an equilibrium measure for $f$.
\end{lemma}

\begin{proof}
Let $X$ have period $p$ and let $X_0, \dots, X_{p-1}$ be the cyclically moving classes of $X$ as in the proof of Proposition \ref{yoo-sv}. Note that each $X_i$, $0 \leq i \leq p-1$ is $\fT_X$-invariant and that we can regard each one as a mixing shift of finite type with alphabet a subset of $\cB_p(X)$, so that it is meaningful to speak of the Gibbs relations $\fT_{X_i}$. Moreover, since each class $X_i$ is $\fT_{X}$-invariant, we have for each $i$ that $\fT_{X_i}$ is simply the restriction of $\fT_X$ to $X_i \times X_i$.

Let $\mu$ be a Gibbs measure for $f$ and  $\mu'$ be the normalized restriction of $\mu$ of $X$ to $X_0$, i.e., $\mu'(E) = p \mu(E)$ for any Borel set $E \subseteq X_0$. Lemma \ref{lemma-period-trans} shows that every holonomy $\psi: A \to B$ of $\fT_X$  restricts to a holonomy $\psi_0 : A \cap X_0 \to B \cap X_0$ of the relation $\fT_{X_0}$. Moreover, every holonomy of $\fT_{X_0}$ clearly arises as such a restriction: if $\psi_0$ is a holonomy of $\fT_{X_0}$, $\psi_0 = \psi|_{X_0}$ where $\psi|_{X_i} = \sigma^i \circ \psi_0 \circ \sigma^{-i}$. Therefore the measure $\mu'$ inherits the Gibbsianness of $\mu$, with respect to $R_p f =  \sum_{j=0}^{p-1} f \circ \sigma^j$. Since each $X_i$, $0 \leq i \leq p-1$ is a mixing shift of finite type, therefore strongly irreducible, Theorem \ref{dobrushin} shows that $\mu'$ is an equilibrium measure for $R_p f$. As in the proof of Proposition \ref{yoo-sv}, we have 
\[
h_X(\mu) + \int_X f \, d\mu = \frac{1}{p} \left( h_{X_0}(\mu') + \int_{X_0} R_p f \, d\mu'         \right),
\]
which concludes that $\mu$ is an equilibrium measure for $f$.

\end{proof}

The next corollary generalizes Proposition \ref{prop-pres-gibbs} from almost invertible codes to finite-to-one codes, in the case of a cocycle induced by a sufficiently regular potential. As discussed in \S \ref{pres-gibbs-sec}, the proof technique that we used for Proposition \ref{prop-pres-gibbs} does not generalize to higher degrees. To obtain Corollary \ref{irred-fin-1}, we apply Theorem \ref{sofic-lf}, even though the statement of Corollary \ref{irred-fin-1} does not mention equilibrium measures explicitly.

\begin{corollary}[preservation of Gibbsianness for finite-to-one factor codes]\label{irred-fin-1}
Let $X$ be an irreducible shift of finite type and $Y$ an irreducible sofic shift. Let $\pi: X \to Y$ be a finite-to-one factor code, let $f\colon Y \to \R$  be a potential with summable variation, that is, $f \in SV(Y)$, and let $\mu$ be a $\sigma$-invariant Gibbs measure for $f \circ \pi$. Then $\nu = \pi_* \mu$ is a Gibbs measure for $f$.
\end{corollary}

\begin{proof}
First note that $\mu$ is an equilibrium measure for $f \circ \pi$ by Lemma \ref{dobrushin-irred-sft}. Since $\pi$ is finite-to-one, $\nu$ is an equilibrium measure for $f$ (\cite{yoo2018decomposition}, Lemma 4.4), so by Theorem \ref{sofic-lf}, $\nu$ is Gibbs for $f$.
\end{proof}

While Corollary \ref{irred-fin-1} may look like it generalizes Proposition \ref{prop-pres-gibbs} in that $\mu$ need no longer be ergodic and $\pi$ need no longer have degree one. However, the cocycle in Proposition \ref{prop-pres-gibbs} need not a priori be induced by the pullback of a potential with summable variation. Corollary \ref{irred-fin-1} therefore generalizes Proposition \ref{prop-pres-gibbs} only in this special case.

\begin{theorem}[Dobrushin theorem for irreducible sofic shifts]\label{sofic-dobrushin}
Let $Y$ be an irreducible sofic shift and let $f\colon Y \to \R$  be a potential with summable variation, that is, $f \in SV(Y)$. Let $\nu$ be a $\sigma$-invariant Gibbs measure for $f$. Then $\nu$ is an equilibrium measure for $f$.
\end{theorem}

\begin{proof}
First suppose that $\nu$ is ergodic. Let $\pi: X \to Y$ be the minimal right-resolving presentation of $Y$. By Lemma \ref{sofic-gibbs-pushfwd}, there exists an ergodic $\sigma$-invariant Gibbs measure $\mu$ for $f \circ \pi$ with $\pi_* \mu = \nu$. By Lemma \ref{dobrushin-irred-sft}, $\mu$ is an equilibrium measure for $f \circ \pi$. Finally, again by Lemma 4.4 in \cite{yoo2018decomposition}, $\nu$ is an equilibrium measure for $f$. The result for general (not necessarily ergodic) $\nu$ follows from the ergodic case by compactness and convexity as in the proof of Theorem \ref{sofic-lf}.
\end{proof}

Taken together, Theorems \ref{sofic-lf} and \ref{sofic-dobrushin} show that for a potential with summable variations on an irreducible sofic shift, the equilibrium measures are precisely the $\sigma$-invariant Gibbs measures.

Finally, we can use Theorem \ref{sofic-dobrushin} to generalize Proposition \ref{sofic-gibbs-pushfwd} from almost invertible to finite-to-one codes, in the same special case for which Corollary \ref{irred-fin-1} generalizes Proposition \ref{prop-pres-gibbs}.

\begin{corollary}[lifting Gibbs measures through almost invertible factor codes]\label{lift-gibbs-fin-1}
Let $X$ be an irreducible shift of finite type and $Y$ an irreducible sofic shift. Let $\pi: X \to Y$ be a finite-to-one factor code, let $f\colon Y \to \R$  be a potential with summable variation, that is, $f \in SV(Y)$, and let $\nu$ be a $\sigma$-invariant Gibbs measure for $f$. Then there exists a $\sigma$-invariant Gibbs measure $\mu$ for $f \circ \pi$ with $\pi_* \mu = \nu$.
\end{corollary}

\begin{proof}
By Theorem \ref{sofic-dobrushin}, $\nu$ is an equilibrium measure for $f$. By Lemma \ref{lift-eqm}, there is an equilibrium measure $\mu$ for $f \circ \pi$ with $\pi_* \mu = \nu$. By Theorem \ref{lf}, $\mu$ is a Gibbs measure for $f \circ \pi$, and is certainly $\sigma$-invariant.
\end{proof}

In closing, we note that it is an open problem to determine the existence of a finite-to-one factor code from a given shift of finite type $X$ onto a given sofic shift $Y$, as in the hypotheses of Corollaries \ref{irred-fin-1} and \ref{lift-gibbs-fin-1}. Equal entropy is necessary, but there are additional necessary conditions; see \cite{lind1995introduction}, \S 12.2.

\section*{Acknowledgments}

We thank Tom Meyerovitch and Brian Marcus for their generous advice throughout the course of this work. We also thank the organizers and participants of the West Coast Dynamics Seminar and the Ottawa Mathematics Conference 2020 for their helpful questions and comments in response to presentations based on earlier versions of this paper.

\end{document}